\documentclass{amsart}

\usepackage{amsmath,amssymb,amsthm}
\usepackage{graphicx,tikz}
\newtheorem{theorem}{Theorem}
\newtheorem*{thm}{Theorem}
\newtheorem*{proposition}{Proposition}

\newtheorem{lemma}{Lemma}

\theoremstyle{definition}

\theoremstyle{remark}

\begin{document}

\title[]{ Polynomials with Zeros on the Unit Circle: Regularity of Leja Sequences}

\author[]{Stefan Steinerberger}
\address{Department of Mathematics, University of Washington, Seattle, WA 98195, USA}
\email{steinerb@uw.edu}
\thanks{The author is supported by the NSF (DMS-1763179) and the Alfred P. Sloan Foundation.}

\begin{abstract} Let $z_1, \dots, z_m$ be $m$ distinct complex numbers, normalized to $|z_k| = 1$, and consider the polynomial
$$ p_{m}(z) = \prod_{k=1}^{m}{(z-z_k)}.$$
We define a sequence of polynomials in a greedy fashion,
$$ p_{N+1}(z) = p_{N}(z) \left(z - z^*\right)\qquad \mbox{where}~z^* = \arg\max_{|z|=1} |p_{N}(z)|,$$
and prove that, independently of the initial polynomial $p_m$, the roots of $p_{N}$ equidistribute in angle at rate at most $(\log{N})^2/N$. This even persists when sometimes adding `adversarial' points by hand. We also obtain sharp rates for an $L^2-$version of a problem first raised by Erd\H{o}s and solved by Beck in $L^{\infty}$.
\end{abstract}

\maketitle

\section{Introduction}
\subsection{Introduction.}   Let $(x_n)_{n=1}^{\infty}$ be a sequence on $[0,1]$. We define the discrepancy function $D_N:[0,1] \rightarrow [0,1]$ associated of the first $N$ elements via
$$ D_N(x) = \left| \frac{1}{N} \#\left\{ 1 \leq k \leq N: x_k \leq x\right\} - x\right|$$
van der Corput \cite{vdc, vdc2} asked in 1935 whether there was a sequence for which $  \| D_N\|_{L^{\infty}} \leq c/N$ for some constant $c$ and all $N \in \mathbb{N}$. This would correspond to an exceptionally regular sequence having the property that its first $N$ elements are, up to a constant, as regularly distributed as possible (uniformly in $N$). It was shown by van Aardenne-Ehrenfest \cite{aard} that such sequences do not exist. This prompts the question: how regular can sequences be?
Improving on work of K. F. Roth \cite{roth}, W. M. Schmidt \cite{sch} established the optimal result stating that for a universal $c>0$ and any sequence $(x_n)_{n=1}^{\infty}$ in $[0,1]$ there are infinitely many $N$ for which
$$\| D_N\|_{L^{\infty}} \geq c \frac{\log{N}}{N}.$$
 Sequences that attain it mainly stem from two types of regular structures:
\begin{enumerate}
\item \textbf{Irrational Rotations.} If $\alpha \in \mathbb{R}$ is badly approximable ($\alpha = \sqrt{2}$ works), then the sequence $x_n = \left\{n \alpha \right\}$, where
$\left\{x \right\} = x - \left\lfloor x \right \rfloor$ is the fractional part, is known to have the optimal growth rate $N \cdot \|D_N\|_{L^{\infty}} \leq c_{\alpha} \log{N}$. 

\item \textbf{Digit Expansions.} The van der Corput sequence has a simple definition: to obtain $x_n$, write the integer $n$ in binary expansion, reverse the order of digits and convert it to a real number in $[0,1]$. This sequence begins
$$ \frac{1}{2}, ~\frac{1}{4},~\frac{3}{4},~\frac{1}{8},~\frac{5}{8},~\frac{3}{8},~\frac{7}{8},~\dots$$
\end{enumerate}
Given the importance of the question (and the usefulness of such sequences), both examples have
been generalized in a large number of directions, we refer to the classical textbooks by Beck \& Chen \cite{bc}, Dick \& Pillichshammer \cite{dick}, Drmota \& Tichy \cite{drmota} and Kuipers \& Niederreiter \cite{kuipers}. 

\subsection{Motivation.} Mathematics is full of regular structures -- it is interesting that the main `sources of regularity' for this problem have been restricted to these two very specific structures. We believe it could be quite interesting to see whether existing mathematical structures could be used to construct new sequences which are highly regular for reasons completely different from the ones we mentioned above.
There is a second motivating factor: while van der Corput's question has been answered for sequences on $[0,1]$, the problem is wide open even on $[0,1]^2$: here, we define the discrepancy function $D_N(\textbf{x})$ in the analogous way by considering the number of elements in the box $[0,\textbf{x}] \subset [0,1]$. Results for this case were obtained by K. F. Roth \cite{roth}, J. Beck \cite{beck}. Bilyk \& Lacey \cite{bil3} and Bilyk, Lacey \& Vagharshakyan \cite{bil4}. What makes the problem of $[0,1]^d$ for $d\geq 2$ particularly interesting is that there  are two competing conjectures (corresponding to different powers of $\log{N}$). 
\begin{itemize}
\item \textbf{Either} we know the most regular sequences in $[0,1]^d$ (with variations of the Kronecker sequence or the van der Corput sequence providing examples) and we simply do not know how to prove that nothing better exists
\item \textbf{or} there are structures more regular than anything that we currently know.
\end{itemize}
Both conjectures have interesting arguments supporting them. Some would argue that sequences and their regularity properties have been studied
for over a century, how likely is it that such extraordinary objects could have been overlooked for so long? The second conjecture is supported by a number of structural similarities between this problem and philosophically related other problems (we refer to the excellent survey \cite{bil} for an in-depth discussion).
Finally, while the problem as phrased above has immediate intrinsic appeal, there are also practical considerations: highly regular sequences are useful in a variety of settings (sampling, interpolation, numerical integration, computer graphics, ...). While sequences constructed via combinatorial or number-theoretic reasoning do indeed perform well on $[0,1]^d$, it is not as clear how one should proceed if one works in a general domain $\Omega \subset \mathbb{R}^d$ or on a manifold $(M,g)$.

\section{Statement of Results}
The main goal of this paper is to discuss a new source of regularity: potential theory. We consider Leja sequences on the unit circle which are given by the following construction: let $z_1, \dots, z_m$ be $m$ distinct complex numbers normalized to $|z_k| = 1$ and consider the polynomial
$$ p_{m}(z) = \prod_{k=1}^{m}{(z-z_k)}.$$
Consider a sequence of polynomials defined in a greedy fashion:
$$ p_{N+1}(z) = p_{N}(z) \left(z - z^*\right)\qquad \mbox{where}~z^* = \arg\max_{|z|=1} |p_{N}(z)|.$$
If the maximum is attained in more than one point, any of the points in which the maximum is attained is admissible. This gives rise to a sequence of polynomials defined by an increasing sequence of roots $(z_n)_{n=1}^{\infty}$. It is known that if the initial polynomial is linear, $p_1(z) = z - z_1$, then the arising sequence of roots can be characterized in terms of the van der Corput sequence: a more general result along these lines is due to Pausinger \cite{pausinger}. To the best of our knowledge, nothing is known when the initial polynomial is of higher order: we show that the mechanism works for any initial polynomial. Moreover, \textit{the underlying mechanism is stable} (see \S 2.1.).

\begin{center}
\begin{figure}[h!]
\begin{tikzpicture}[scale=2]
\node at (-2,0) {};
\draw[thick] (0,0) circle (1cm);
\filldraw (0.523, 0.851) circle (0.025cm);
\filldraw (-0.367, 0.93) circle (0.025cm);
\filldraw (-0.971, 0.239) circle (0.025cm);
\filldraw (-0.964, -0.265) circle (0.025cm);
\filldraw (-0.666, -0.745) circle (0.025cm);
\filldraw (-0.202, -0.979) circle (0.025cm);
\filldraw (0.961, -0.273) circle (0.025cm);
\filldraw (1, 0) circle (0.025cm);
\draw [thick, ->] (1.3,0.75) -- (1.05, 0.7);
 \node at (1.7, 0.9) {new point here};
\draw [thick] (0.809, 0.587) circle (0.03cm);
\end{tikzpicture}
\caption{Given a set of points, we add the new point that maximizes the product of the distances to the previous points. }
\end{figure}
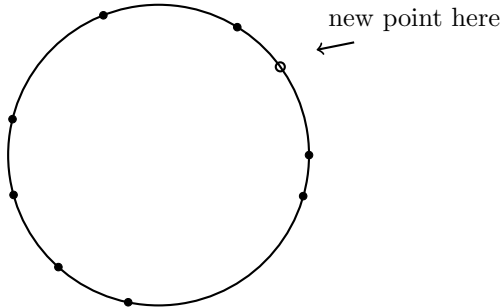
\end{center}
We will, throughout the paper, work on $[0,1]$ after transforming variables. Writing complex numbers of size 1 as
$ z_k = e^{2 \pi i x_k},$
we have
$$ \log  \left( \prod_{j=1}^{N}{\| e^{2\pi ix}- z_j\|}\right) = \sum_{j = 1 }^{N}{\log{|2\sin{(\pi |x-x_j|)}|}}.$$
At each step, we pick the next point $x_{N+1}$ to maximize this expression (see Fig. 2).  This function has logarithmic singularities in $x_1, \dots, x_N$ where it tends to $-\infty$. The underlying idea is as follows: each existing point $x_1, \dots, x_N$ can be thought of as a particle contributing to a global `energy field'. The field is quite negative very close to the particle (certainly the product of all the distances is quite small when one is very close to one of the existing points). 
 The new point has the largest `energy' and is thus `the furthest away'. Of course things are not quite as simple: which functions (or `energy fields') will work and \textit{why} do they work? We refer to Fig. 1 and Fig. 2 for examples and to \S 2.5. and \cite{brown} for the bigger question.\\

 We first discuss a variation of Leja sequences which we call 'Symmetric Leja sequences'. The advantage of this construction is that there is an additional degree of symmetry which allows us to phrase all the regularity statements in terms of the classical discrepancy function $D_N$. This will be carried out in \S 2.1, in \S 2.2 we discuss an associated potential-theoretic estimate due to G. Wagner. \S 2.3 discusses the general case for which we obtain a regularity statement in terms of, a Fourier-analytic measure of regularity, the \textit{diaphony}
$$ \|F_N\|_{L^2}^2 = \frac{1}{2\pi^2} \sum_{k=1}^{\infty} \frac{1}{k^2}  \left| \frac{1}{N} \sum_{k=1}^{N} e^{2 \pi i \ell x_k} \right|^2.$$
An Erd\H{o}s problem is discussed in \S 2.4. We discuss what we consider to be one of the main open problems in \S 2.5, \S.2.6 discusses related results.

\begin{center}
\begin{figure}[h!]
\begin{tikzpicture}[scale=5]
\node at (0.5,0.4) {\includegraphics[width=0.414\textwidth]{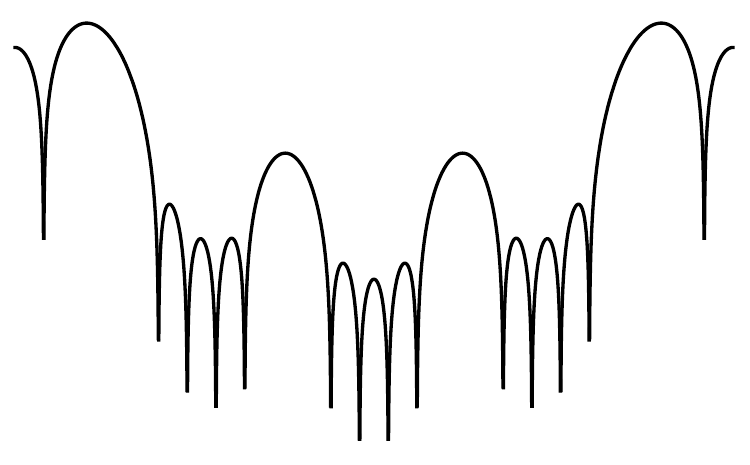}};
\node at (-0.5,0) {};
\draw (0,0) -- (1,0);
\filldraw (0.04, 0) circle (0.01cm);
\filldraw (0.2, 0) circle (0.01cm);
\filldraw (0.24,0, 0) circle (0.01cm);
\filldraw (0.28, 0) circle (0.01cm);
\filldraw (0.32, 0) circle (0.01cm);
\filldraw (0.44, 0) circle (0.01cm);
\filldraw (0.48, 0) circle (0.01cm);
\filldraw (0.52, 0) circle (0.01cm);
\filldraw (0.56, 0) circle (0.01cm);
\filldraw (0.68, 0) circle (0.01cm);
\filldraw (0.72, 0) circle (0.01cm);
\filldraw (0.76, 0) circle (0.01cm);
\filldraw (0.8, 0) circle (0.01cm);
\filldraw (0.96, 0) circle (0.01cm);
\draw[thick] (0.9, 0) circle (0.012cm);
\node at (1.3, 0.8) {maximum here};
\draw [thick, ->] (1.3,0.75) -- (1.05, 0.7);
\draw[thick] (0.9, 0.69) circle (0.012cm);
\draw [ ->] (0.9,0.6) -- (0.9, 0.1);
\node at (0.9, -0.1) {new point here};
\end{tikzpicture}
\caption{Given a set of points, find a place where the function assumes a maximum and add the location as a new point. }
\end{figure}
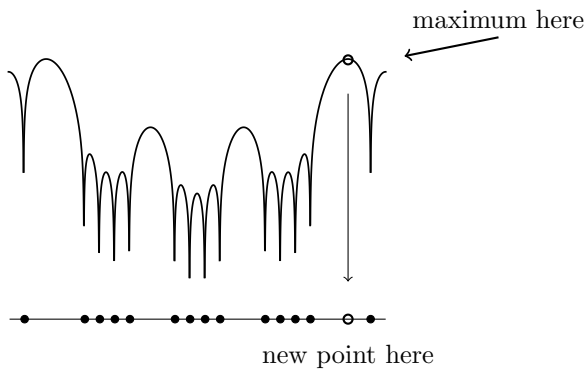
\end{center}

\subsection{Symmetric Leja Sequences}
Let $z_1, \dots, z_m$ be $m$ distinct complex numbers normalized to $|z_k| = 1$ and consider the polynomial
$$ p_{2m}(z) = \prod_{k=1}^{m}{(z-z_k)(z-\overline{z_k})}.$$
We introduce a sequence of polynomials in a greedy fashion: 
\begin{align*}
 p_{2N+1}(z) &= p_{2N}(z) \left(z - z^*\right)\qquad \mbox{where}~z^* = \arg\max_{|z|=1} |p_{2N}(z)| \\
p_{2N+2}(z) &= p_{2N+1}(z) \left(z-\overline{z^*}\right) \quad \mbox{where}~z^* = \arg\max_{|z|=1} |p_{2N}(z)|.
 \end{align*}
The polynomial $p_{N}$ will then have $N$ roots of which at least the first $N-1$ come in symmetric pairs. 
We now state the main result: the polynomial $p_{N}$ has $N$ roots $z_1, \dots, z_N$ on the unit circle to which we associate $N$ real numbers $x_1, \dots, x_N$ via
$$ z_k = e^{2\pi i x_k}.$$
The way the polynomials are constructed, we obtain an infinite sequence $(z_n)_{n=1}^{\infty}$ of roots on the unit circle and an associated infinite sequence $(x_n)_{n=1}^{\infty}$ in $[0,1]$.

\begin{theorem}
The sequence $(x_n)_{n=1}^{\infty}$ satisfies
$$ \frac{1}{N}\sum_{n=1}^{N} n\|D_n\|_{L^1} \lesssim  (\log{N})^2.$$
Here, the implicit constant is universal once $N$ is sufficiently large.
\end{theorem}
This argument shows, that, for typical values of $N$, we have $\|D_N\|_{L^1} \lesssim N^{-1}(\log{N})^2$ and there cannot be too many exceptions.
This is optimal up to logarithmic factors, it is known that the best uniform rate of equidistribution is $\|D_N\|_{L^1} \lesssim N^{-1}\sqrt{\log{N}}$. It seems reasonable
to assume that this is indeed the correct rate and that our additional factor of $(\log{N})^{3/2}$ is an artifact of the proof. Likewise, it seems reasonable to conjecture that for such sequences we also have $\|D_N\|_{L^{\infty}} \lesssim N^{-1} \log{N}$.\\

\textit{Remark.} The proof also shows something else that is a bit harder to make explicit but easy enough to explain: if we compute the first $N$ elements of the sequence and then, say, manually add $\sim \mathcal{O}(1)$ arbitrary new points of our own choosing (as long as they are distinct from the existing points) before resuming the greedy construction, this does not change the asymptotic behavior. Or, alternatively, if we compute the first $N$ elements of the sequence and then manually alter $\sim \mathcal{O}(1)$ elements of those first $N$ elements in such a way that all points remain distinct this still does not impact the asymptotic behavior: \textit{the mechanism is stable and automatically adjusts}!

\subsection{An Estimate of Wagner.}
A crucial ingredient in our proof is an estimate of Wagner. We first state Wagner's estimate \cite{wagner} in a somewhat specialized setting (which allows us to phrase it in terms of the discrepancy). It states that the regularity of a set of points $\left\{x_1, \dots, x_N\right\} \subset [0,1]$ can be detected by looking at the size of the polynomial
$$ p_N(z) = \prod_{k=1}^{N}{ (z - e^{2\pi i x_k})}.$$
If the set of points is not very evenly distributed, then the $L^1-$norm of the polynomial restricted to the unit circle $\| \log p_N(e^{it})\|_{L^1}$ will be large.
Conversely, if $\| \log p_N(e^{it})\|_{L^1}$ is small, then the set $\left\{x_1, \dots, x_N\right\}$ has to be evenly distributed.
\begin{thm}[Wagner \cite{wagner}] Let $\left\{x_1, \dots, x_N\right\}$ be a symmetric set, i.e. $N$ is even and $x_{2k+2} = 1-x_{2k+1}$. Then
$$ \int_{0}^{1} \left| \sum_{ k=1}^{N} \log \left|2 \sin{(\pi (x - x_k))} \right| \right| dx \gtrsim \frac{N }{\log{N}}\| D_N\|_{L^1}.$$ 
\end{thm}
It is not clear whether the factor $\log{N}$ in this inequality can be removed. Wagner remarks that `[...] holds even if the factor $1/\log{N}$ is omitted, but we could not prove this in general' \cite{wagner}.   This would be interesting to know: in particular, it would lead to a logarithmic improvement of Theorem 1 and Theorem 2. We remark that by increasing the $L^1-$norm to $L^2$, we obtain an identity.
\begin{proposition} Let $\left\{x_1, \dots, x_N\right\}$ be a symmetric set, i.e. $N$ is even and $x_{2k+2} = 1-x_{2k+1}$. Then
$$ \int_{0}^{1} \left| \sum_{ k=1}^{N} \log \left|2 \sin{(\pi (x - x_k))} \right| \right|^2 dx  = \pi^2 N^2 \| D_N\|^2_{L^2}.$$ 
\end{proposition}
We mention a more general version (not requiring symmetry) in the next section; this will  be useful in the analysis of an Erd\H{o}s problem in \S 2.4.

\subsection{General Leja points.} We can phrase all our results without imposing symmetry on the points when resorting to a different measure of regularity: let $f:\mathbb{R} \rightarrow \mathbb{R}$ denote the one-periodic function
$$ f(x) = \frac{1}{2} - \left\{x \right\},$$
where $\left\{x\right\} = x - \left\lfloor x \right\rfloor$. We define 
$$F_N(x) = \frac{1}{N}\sum_{k=1}^{N}{f(x-x_k)}$$
and use $\|F_N\|_{L^2}$ as a measure of regularity of the point set. Note that
$$ F_N(x) = \frac{i}{2\pi} \sum_{\ell \in \mathbb{Z} \atop \ell \neq 0} \frac{1}{\ell} \left| \frac{1}{N} \sum_{k=1}^{N} e^{-2 \pi i \ell x_k} \right| e^{2\pi i \ell x}$$
and thus
$$ \|F_N\|_{L^2}^2 = \frac{1}{2\pi^2} \sum_{k=1}^{\infty} \frac{1}{\ell^2}  \left| \frac{1}{N} \sum_{k=1}^{N} e^{2 \pi i \ell x_k} \right|^2$$
is, up to universal constants, comparable to the diaphony (alternatively, the $\dot H^{-1}$ norm of the measure given by the sum of Dirac measures placed in the point set).
We then proceed as above but without the additional symmetry. Let $z_1, \dots, z_m$ be $m$ distinct complex numbers normalized to $|z_k| = 1$ and consider the polynomial
$$ p_{m}(z) = \prod_{k=1}^{m}{(z-z_k)}.$$
We introduce a sequence of polynomials in a greedy fashion: 
$$ p_{N+1}(z) = p_{N}(z) \left(z - z^*\right) \qquad \mbox{where} \quad z^* = \arg\max_{|z|=1} |p_{N}(z)|.$$
If the maximum is not unique, any choice of location of maximum is admissible.
 $p_{N}$ has $N$ roots $z_1, \dots, z_N$ to which we associate $N$ real numbers $x_1, \dots, x_N$ via
$$ z_k = e^{2\pi i x_k}.$$
The way the polynomials are constructed, we obtain an infinite sequence $(z_n)_{n=1}^{\infty}$ of roots on the unit circle and an associated infinite sequence $(x_n)_{n=1}^{\infty}$ in $[0,1]$.

\begin{theorem}
The sequence $(x_n)_{n=1}^{\infty}$ satisfies
$$ \frac{1}{N}\sum_{n=1}^{N} n\|F_n\|_{L^1} \lesssim  (\log{N})^2.$$
Here, the implicit constant is universal once $N$ is sufficiently large.
\end{theorem}
This implies that for a typical value of $N$ we have $\|F_N\|_{L^1} \lesssim (\log{N})^2/N$. This is at most a factor of $(\log{N})^{3/2}$ from
optimal and it seems reasonable to believe that this loss of logarithms is an artifact of the proof. We conclude by remarking that for any $\left\{x_1, \dots, x_N\right\} \subset [0,1]$,
$$ \int_{0}^{1} \left|\frac{1}{N} \sum_{ k=1}^{N} \log \left|2 \sin{(\pi (x - x_k))} \right| \right|^2 dx = \pi^2  N^2\| F_N\|^2_{L^2}  .$$

\subsection{A Problem of Erd\H{o}s.} A problem of Erd\H{o}s \cite{erdos} is as follows: let $(z_n)_{n=1}^{\infty}$ be an infinite sequence of complex numbers on the unit circle, $|z_k| = 1$, and define the associated sequence
$$ a_N = \max_{|z|=1} \prod_{k=1}^{N}{|z-z_k|}.$$
Can the sequence $a_N$ be uniformly bounded in $N$? Hayman (see \cite{erd2}) observed that there is a sequence $(z_n)_{n=1}^{\infty}$ such that
$a_N \leq N$ and this was improved by Linden \cite{linden} to $a_N \leq N^{1-\delta}$.
The question was then answered by Wagner \cite{wagner0} who showed that
$$ \max_{1 \leq n \leq N} a_n \gtrsim (\log{N})^c,$$
where $c>0$ is an absolute constant. Beck \cite{beck2} obtained the optimal rate
$$ \max_{1 \leq n \leq N} a_n \gtrsim N^c.$$
Rephrased in a different notation, for any infinite sequence $(z_n)_{n=1}^{\infty}$ of normalized complex numbers, $|z_k|=1$, Beck's result \cite{beck2} can be phrased as
$$ \left\| \log{ \prod_{k=1}^{N}{(z-z_k)} } \right\|_{L^{\infty}(|z|=1)} \gtrsim \log{N} \qquad \mbox{for infinitely many}~N.$$
Moreover, this is the optimal rate. We can prove an analogous version when the size of the polynomial $p_N$ is being measured in $L^2$ instead of $L^{\infty}$.

\begin{theorem} For any infinite sequence $(z_n)_{n=1}^{\infty} $ with $|z_k|=1$, we have
$$ \left\| \log{ \prod_{k=1}^{N}{(z-z_k)} } \right\|_{L^{2}(|z|=1)}   \gtrsim \sqrt{ \log{N}} \qquad \emph{for infinitely many}~N.$$
\end{theorem}
This is the optimal rate (up to the value of the implicit constant).

\subsection{Open Questions}
There are several questions regarding the optimality of these results (say, whether it is possible to improve Theorem 1 and 2 by a logarithmic factor). However, the main question seems to be the following
\begin{quote}
\textbf{Question.} On what domains can potential-theoretic sequences be used to obtain regularly distributed sequences?
\end{quote}
Leja sequences in subsets $K \subset \mathbb{C}$ in the complex plane are defined via
$$ z_{N+1} = \arg \max_{z \in K} \prod_{k=1}^{N}{|z-z_k|}.$$
We do not necessarily expect this definition to be very promising in terms of uniform distribution (on the unit interval, for example, their limiting density is not uniform \cite{reichel}). However, as is the main point of our paper, they do indeed result in highly regular sequences when constructed on the unit circle. What is special about the unit circle? Perhaps it is the following algebraic `coincidence': when re-parametrized to the unit interval, the definition of the sequence is
$$ x_{N+1} = \arg \max_{0 \leq x \leq 1}  \sum_{k=1}^{N}\log{|2\sin{(\pi |x-x_k|)}|}.$$
At the same, this mysterious function $\log{(2\left|\sin{(\pi x)}\right|)}$ arises naturally as an infinite Fourier series
$$ \sum_{k=1}^{\infty}{\frac{\cos{(2 \pi k x)}}{k}} = -\log{(2\left|\sin{(\pi x)}\right|)}.$$
This infinite Fourier series, in turn, corresponds to the definition of the inverse fractional Laplacian (possibly up to constants) on $[0,1]$ since
$$ \left[ (-\Delta)^{-\frac{1}{2}} \delta_x\right](y)=  \sum_{k=1}^{\infty}{\frac{\cos{(2 \pi k(x-y))}}{k}}.$$
At this point, one could wonder: is maybe the inverse Laplacian doing all the work? The fractional power seems unusual, what, for example, if one were to choose the standard inverse Laplacian instead? This is
$$ \left[ (-\Delta)^{-1} \delta_x\right](y)=  \sum_{k=1}^{\infty}{\frac{\cos{(2 \pi k (x-y))}}{ (2 \pi k)^2}} = \frac{1}{4} \left[ |x-y|^2 - |x-y| + \frac{1}{6} \right],$$
where the polynomial is the second Bernoulli polynomial. Brown and the author \cite{brown} have investigated this problem. Numerics suggest that there is very little difference in the behavior of the sequences: the sequence
$$ x_{N+1} = \arg \min_{x}  \left( (- \Delta)^{-1}  \sum_{k=1}^{N}{\delta_{x_k}} \right)(x)$$
seems to also enjoy good distribution properties. \cite{brown} also established some sharp optimality results in terms of Wasserstein distance in dimensions $d\geq 3$ on general compact manifolds -- however, the behavior in terms of discrepancy remains poorly understood (even for $(-\Delta)^{-1}$, given by a quadratic polynomial, on the unit interval). This naturally leads one to wonder whether such potential-theoretic constructions have a chance of having optimal rate of regularity. 
\subsection{Related results.}
 Leja sequences first arose in the work of Edrei \cite{edrei} and, independently, Leja \cite{leja}. They are simply defined in a greedy fashion by having the next element maximize the product of the distances to the existing elements. They are known to be excellent points for polynomial interpolation of analytic functions.
To the best of our knowledge, the existing results on $\left\{z \in \mathbb{C}: |z|=1\right\}$ are usually stated only for the case where a Leja sequence is initialized with a single element $x_1$.
  Bialas-Ciez \& Calvi \cite{ci} proved that if one constructs such a sequence from a single initial element, then the arising sequence is highly structured. Pausinger \cite{pausinger} characterized the arising sequences (since there are often several maxima, there is an ambiguity in which one to pick and one can obtain several sequences); he also showed that this characterization also holds true for a larger class of notions of energy. Lopez-Garcia \& Wagner \cite{lopez} established energy asymptotics.
G\"otz \cite{gotz}, building on earlier machinery of Andrievski \& Blatt  \cite{and1,and2}, Blatt \cite{blatt}, Blatt \& Mhaskar \cite{blatt2} and Totik \cite{totik}, proved that $\|D_N\|_{L^{\infty}} \lesssim N^{-1/2} \log{N}$. The author (unknowingly) recovered this bound in a different setting \cite{steini} (the guiding motivation was to interpret the Erd\H{o}s-Turan inequality as an energy functional; this perspective was also useful in \cite{stein}). A philosophically related object was studied in \cite{aist, hla1, hla2}.
We also refer to Baglama, Calvetti \& Reichel \cite{baglama}, Lopez-Garcia \& Saff \cite{lopez0}, Pritsker \cite{pritsker, pritsker2} and \cite{and3, hardin, saff, wagner1, wagner2}.\\

\section{Proofs}
\subsection{Two Lemmata.}

\begin{lemma} Let $\left\{x_1, \dots, x_n\right\} \subset [0,1]$. Then 
$$ \sum_{i,j=1 \atop i \neq j }^{N}{\log{|2\sin{(\pi |x_i-x_j|)}|}} \leq N \log{N}.$$
\end{lemma}
\begin{proof} We introduce
$$ z_k = e^{2 \pi i x_k}$$
which allows us to write
$$\sum_{i,j=1 \atop i \neq j }^{N}{\log{|2\sin{(\pi |x_i-x_j|)}|}} = \log  \left( \prod_{i,j=1 \atop i \neq j}^{N}{\|z_i - z_j\|}\right).$$
Maximizing this quantity is the famous problem of Fekete \cite{fekete}: on the circle, it is known to be maximal when the $z_i$ are $N-$th roots of unity. In that case, there exists an amusing identity
$$ \prod_{k=1}^{N-1} \sin\left(\frac{\pi k}{N}\right) = \frac{2N}{2^N}$$
from which the bound follows.
\end{proof}

We introduce the 1-periodic function 
$$ f(x) = 1/2 - \left\{x\right\}.$$
\begin{lemma} Let $\left\{x_1, \dots, x_N\right\}$ be a symmetric set, i.e. $N$ is even and $x_{2k+2} = 1-x_{2k+1}$. Then, for $0<x<1$,
$$ \sum_{k=1}^{N}f(x-x_k) = \# \left\{ 1\leq n \leq N: x_k \leq x\right\} - Nx.$$
\end{lemma}
\begin{proof} We have, for any $0 \leq a \leq 1/2$,
$$ f(x-a) + f(x-(1-a)) = \begin{cases} - 2x \qquad &\mbox{if}~x \leq a \\
1-2\left\{a\right\}  - 2\left(x-1-2\left\{a\right\} \right) \qquad &\mbox{if}~a \leq x \leq 1-a\\
2a - 2(x-(1-a))  \qquad &\mbox{if}~1-a \leq x \leq 1.
\end{cases}$$
This can be rewritten in terms of characteristic functions as
$$ f(x-a) + f(x-(1-a)) = \chi_{a < x} + \chi_{1-a<x} - 2x.$$
The Lemma then follows by summation over all pairs.
\end{proof}

\subsection{Proof of Wagner's estimate} We include a proof for the convenience of the reader: there is a small gap in the original argument\footnote{I am grateful to Alex Cohen for bringing this to my attention.} and we present a slight modification that is in the same spirit as the original proof.\\

We introduce the hat function $h$ on the interval $[0, 1/(48N)]$ via
$$ h(x) = \begin{cases} 96 N x \qquad &\mbox{if}~0\leq x \leq \frac{1}{96 N} \\
2-96 N x \qquad &\mbox{if}~\frac{1}{96N} \leq x \leq \frac{1}{48 N} \\
0 \qquad &\mbox{otherwise} \end{cases}$$
and on any other interval of length $1/(48N)$ via translation. We also recall the definition of $ f(x) = 1/2 - \left\{x\right\}$. For any set of $N$ points, we now set
$$ F_N(x) = \frac{1}{N}\sum_{k=1}^{N}f(x-x_k).$$
We recall that $F_N$ is a function having constant slope $-1$ that is interrupted by jumps of size $1/N$ at the points. We also note that it has mean value 0 and that it cannot be too small since
$\|F_N\|_{L^1} \geq 1/(4N)$.

\begin{lemma} Let $\left\{x_1, \dots, x_N \right\}$ be fixed. Let us partition the unit interval $[0,1]$ into $48\cdot N$ intervals of equal length $I_1, I_2, \dots, I_{48N}$. There exists a universal constant $c>0$ such that there always exists a subset $A \subset \left\{1, \dots, 48N\right\}$ of intervals such that
$$ \sum_{a \in A} \int_{I_a} F_N(x) h\left(x-\frac{a-1}{48N}\right) dx \geq c \|F_N\|_{L^1}.$$
\end{lemma}
\begin{proof} Since $F_N$ has mean value 0, we have
$$ \|F_N^+\|_{L^1} = \|F_N^{-}\|_{L^1}$$
and it suffices to capture a sufficient amount of $\|F_N^+\|$.
We subdivide the $48 \cdot N$ intervals into three groups. 
\begin{align*}
 \mathcal{A} &= \mbox{intervals where the value of}~F_N~\mbox{at the left endpoint} \geq \frac{1}{48N}\\
 \mathcal{B} &= \mbox{intervals}~I_j~\mbox{not in}~A~\mbox{for which}~ 48N\int_{I_J}{F_N^+(x)dx} \leq \frac{1}{24N}\\
  \mathcal{C} &= \mbox{intervals}~I_j~\mbox{not in}~A~\mbox{for which}~ 48N\int_{I_J}{F_N^+(x)dx} \geq \frac{1}{24N}.
\end{align*}
We will show that by putting hat functions in $\mathcal{A}$, we capture a positive proportion of the mass.
We will not try to capture the mass in $\mathcal{B}$. Note that
$$ \sum_{b \in \mathcal{B}} \int_{I_b}{F_N^+(x) dx} \leq \frac{1}{24N} \leq \frac{\|F_N^+\|_{L^1}}{3},$$
so there is at most a third of the total positive mass and we can afford not capturing it as long as we capture a sufficient amount of the remaining mass. The
mass in intervals $I_j \in \mathcal{A}$ is relatively easy to capture: we note that $F_N$ is positive in the entire interval and decays at most at slope $-1$. The only way the inner product against the hat function would not capture a constant proportion is if $F_N$ were to increase rapidly in regions where the hat function is small (i.e. close to the right-hand side of the interval which could happen if there are several points there). In that case, however, we have $I_{j+1} \in A$ and capture a fraction of the mass by the hat function in the first half of the next interval. Let now $I_j \in \mathcal{C}$. We observe that
$$ m_j = \max_{x \in I_j} F_N(x) \geq 48N \int_{I_j} F_N^+(x) dx \geq \frac{1}{24N}.$$
Since the slope is $-1$, this shows that $I_{j+1} \in \mathcal{A}$. Moreover, we have
$$ \int_{I_j}F_N^+(x) dx \leq \frac{m_j}{48N}$$
as well as
\begin{align*}
 \int_{I_{j+1}} F_N^+(x) dx &\geq \int_{0}^{1/(48N)} \left(m_j - \frac{1}{48N} -  x\right) dx \\
 &\geq  \int_{0}^{1/(48N)} \left(\frac{m_j}{2} -  x\right) dx = \frac{m_j}{96N} - \frac{1}{96N}\cdot\frac{1}{48N}\\
 &= \frac{1}{96N}\left( m_j - \frac{1}{48N} \right) \geq \frac{m_j}{192 N}
 \end{align*}
 which shows that
 $$ \int_{I_j}F_N^+(x) dx \leq 4\int_{I_{j+1}}F_N^+(x) dx.$$
 This shows that at least a fixed proportion of $\|F_N^+\|_{L^1}$ has to lie in the intervals $\mathcal{A}$ and of those we capture at least a fixed amount concluding the argument.
\end{proof}

\begin{proof}[Proof of Wagner's estimate (summarized from \cite{wagner})] The argument initially follows Wagner's argument. Towards the end, we see that the argument is strong enough to also prove 
$$ \int_{0}^{1} \left|\frac{1}{N} \sum_{ k=1}^{N} \log \left|2 \sin{(\pi (x - x_k))} \right| \right|^2 dx \gtrsim  N^2\| F_N\|^2_{L^2}  .$$

Let $h$ be a hat function supported on an interval of length $(48 N)^{-1}$
$$ h(x) = \begin{cases} 96 N x \qquad &\mbox{if}~0\leq x \leq \frac{1}{96 N} \\
2-96 N x \qquad &\mbox{if}~\frac{1}{96N} \leq x \leq \frac{1}{48 N} \\
0 \qquad &\mbox{otherwise.} \end{cases}$$
We will now put such a hat function in each good interval $I_a$ where $a \in A$ by setting
$$H(x) = \sum_{a \in \mathcal{A}} h\left(x - \frac{a}{N} \right).$$
 The Lemma above shows that, for some universal $c>0$,
\begin{align*}
\int_{0}^{1} F_N(x) H(x) dx &\geq  c\|F_N\|_{L^1}.
\end{align*}
Now we aim to provide an upper bound on the integral. This is done by introducing conjugate functions. Using Lemma 2, we have
$$ N \cdot F_N(x) = \sum_{k=1}^{N}f(x-x_k).$$
The Fourier series of $f$ has a simple closed form resulting in
$$ N \cdot F_N(x) = \frac{1}{\pi}  \sum_{k=1}^{N} \sum_{n=1}^{\infty} \frac{\sin{(2\pi n (x-x_k))}}{n}.$$
Likewise, we have, using the Fourier series of $\log{|2\sin{(\pi |x_i-x_j|)}|}$ that
$$ \sum_{ k=1}^{N} \log \left|2 \sin{(\pi (x - x_k))} \right| =  \frac{1}{\pi}  \sum_{k=1}^{N} \sum_{n=1}^{\infty}  \frac{\cos{(2\pi n (x-x_k))}}{n}.$$
These two Fourier series are conjugate. In particular, if we associate to an arbitrary function
$$ g(x) = a_0 + \sum_{n=1}^{\infty} \left(a_n \cos{(2\pi n x)} + b_n \cos{(2 \pi n x)} \right)$$
the conjugate function
$$ \tilde g (x) =  \sum_{n=1}^{\infty}\left( -b_n \cos{(2\pi n x)} + a_n \cos{(2 \pi n x)}\right),$$
then there is the identity
$$ N\int_{0}^{1}{ F_N(x) g(x) dx} = \int_{0}^{1} \left( \sum_{ k=1}^{N} \log \left|2 \sin{(\pi (x - x_k))} \right|  \right) \tilde g(x) dx.$$
We apply this identity to obtain
\begin{align*}
N \frac{\|F_N\|_{L^1}}{64} &\leq N\int_{0}^{1} F_N(x) H(x) dx =\int_{0}^{1} \left( \sum_{ k=1}^{N} \log \left|2 \sin{(\pi (x - x_k))} \right|  \right) \tilde H(x) dx\\
 &\leq \| \tilde H(x) \|_{L^{\infty}} \int_{0}^{1} \left| \sum_{ k=1}^{N} \log \left|2 \sin{(\pi (x - x_k))} \right|  \right|  dx.
 \end{align*}
 In particular, Wagner's estimate follows from showing that $\| \tilde H(x)\|_{L^{\infty}} \lesssim \log{N}$ which he shows via explicit computation.
 \end{proof}
  However, we can also deduce
 $$  N\frac{\|F_N\|_{L^1}}{64}  \leq  \| \tilde H(x) \|_{L^{2}} \left(\int_{0}^{1} \left| \sum_{ k=1}^{N} \log \left|2 \sin{(\pi (x - x_k))} \right|  \right|^2  dx\right)^{\frac{1}{2}}.$$
 As can be seen using the explicit formula for Fourier series, conjugation of a function does not increase the $L^2-$norm, therefore
 $$ \| \tilde H(x) \|_{L^{2}} \leq \| H(x) \|_{L^{2}} \lesssim 1.$$
 We note that this last inequality is actually an identity since
 $$ \sum_{k=1}^{N}{ \log{|2 \sin{( \pi (x-x_k))}|}} = -\sum_{m =1}^{\infty} \frac{\cos{(2\pi m x)}}{m} \sum_{k=1}^{N} e^{2\pi i k x_n}$$
from which we deduce
$$ \int_{0}^{1} \left| \sum_{k=1}^{N}{ \log{|2 \sin{( \pi (x-x_k))}|}}  \right|^2 dx = \frac{1}{2}\sum_{m=1}^{\infty} \frac{1}{m^2} \left| \sum_{k=1}^{N} e^{2\pi i k x_n} \right|^2.$$

\textbf{Remark.} For our application, it would actually suffice to have to have a lower bound the maximum of 
$$\sum_{ k=1}^{N} \log \left|2 \sin{(\pi (x - x_k))} \right|  .$$
We could emulate Wagner's argument until we obtain
\begin{align*}
  N\frac{\|F_N\|_{L^1}}{64} &\leq \int_{0}^{1} \left( \sum_{ k=1}^{N} \log \left|2 \sin{(\pi (x - x_k))} \right|  \right) \tilde H(x) dx \\
 &\leq \| \tilde H(x) \|_{L^1} \left\|  \sum_{ k=1}^{N} \log \left|2 \sin{(\pi (x - x_k))} \right| \right\|_{L^{\infty}}
  \end{align*}
  which will certainly not result in a worse results. The problem is that we need a lower bound on the maximum and not
  on the $L^{\infty}-$norm: for a function with mean value 0, the $L^1-$norm then certainly as such a lower bound. A
  better understanding of this situation would be desirable.

\subsection{Proof of Theorem 1 and Theorem 2}
\begin{proof} We first describe the proof of Theorem 2 for Leja sequences. Afterwards we explain which modifications are required for symmetric Leja sequences which will then establish Theorem 1.
Recall that 
$$ \sum_{i,j=1 \atop i \neq j }^{N}{\log{|2\sin{(\pi |x_i-x_j|)}|}} \leq N \log{N}.$$
We write
$$ \sum_{i,j=1 \atop i \neq j }^{N}{\log{|2\sin{(\pi |x_i-x_j|)}|}} = 2\sum_{n=2}^{N} \sum_{j=1}^{n-1}\log{|2\sin{(\pi |x_n-x_j|)}|}.$$
The inner sum is easy to analyze: since $x_n$ is chosen so as to maximize this expression and the function has mean value 0, we have
\begin{align*}
\sum_{j=1}^{n-1}\log{|2\sin{(\pi |x_n-x_j|)}|} &=\max_{0 \leq x \leq 1}  \sum_{j=1}^{n-1}\log{|2\sin{(\pi |x -x_j|)}|} \\
&\geq  \frac{1}{2}\left\| \sum_{j=1}^{n-1}\log{|2\sin{(\pi |x -x_j|)}|} \right\|_{L^{1}}.
\end{align*}
Now we employ Wagner's estimate and argue that
$$   \left\| \sum_{j=1}^{n-1}\log{|2\sin{(\pi |x -x_j|)}|} \right\|_{L^{1}} \gtrsim \frac{n-1}{\log{(n-1)}} \| F_{n-1}\|_{L^1} \geq \frac{n-1}{\log{(N)}} \| F_{n-1}\|_{L^1}.$$
From this, Theorem 2 follows. As for Theorem 1, we simply observed that the same argument applies and
$$   \left\| \sum_{j=1}^{n}\log{|2\sin{(\pi |x -x_j|)}|} \right\|_{L^{1}} \gtrsim n \| F_{n}\|_{L^1}$$
when $n$ is even. However, employing Lemma 2, we see that for even $n$, the set is symmetric, by construction,
and thus
$$ n \| F_{n}\|_{L^1} = n \| D_{n}\|_{L^1}.$$
This gives us the desired summand for $n$ is even but not for $n$ is odd. We note that
\begin{align*}
 \|D_N - D_{N+1}\|_{L^1} &\leq \|D_N - D_{N+1}\|_{L^{\infty}} \\
 &= \left\| \frac{\#\left\{ 1 \leq k \leq N: x_k \leq x\right\}}{N}  - \frac{\#\left\{ 1 \leq k \leq N+1: x_k \leq x\right\}}{N+1} \right\|_{L^{\infty}}\\
 &\leq \max_{1 \leq m \leq N} \max\left( \frac{m}{N} - \frac{m}{N+1}, \frac{m}{N} - \frac{m+1}{N+1} \right) \lesssim \frac{1}{N}.
 \end{align*}
This shows that $D_N$ for $N$ odd is not too different from $D_N$ with $N$ even and since we have the identity for even numbers, Theorem 1 follows.
\end{proof}

\subsection{Proof of Theorem 3}

\begin{proof}
We recall that
$$  \frac{1}{4\pi} \int_{|z|=1} \log \left(\prod_{k=1}^{n}{|z-z_k|^2}\right) d \sigma  =\int_{0}^{1} \left|\sum_{k=1}^{n}{\log{|e^{2\pi ix}-e^{2\pi i x_k}|}}\right|^2 dx.
$$
Since
$$ \log{|e^{2 \pi ix}-e^{2 \pi i x_k}|} = \log{|2 \sin{( \pi (x-x_k))}|}$$
and thus
$$ \int_{0}^{1} \left|\sum_{k=1}^{N}{\log{|e^{2\pi ix}-e^{2\pi i x_k}|}}\right|^2 dx =  \int_{0}^{1} \left|\sum_{k=1}^{N}{ \log{|2 \sin{( \pi (x-x_k))}|}} \right|^2 dx.$$
By the equivalence of $F_N$ and the logarithmic potential
$$ \int_{0}^{1} \left|\sum_{k=1}^{n}{ \log{|2 \sin{( \pi (x-x_k))}|}} \right|^2 dx \sim  N^2\| F_N\|^2_{L^2}.$$ 
The remaining ingredient is a classical irregularities of distribution result due to Proinov \cite{proinov} (a nice exposition of the result is due to Kirk \cite{kirk}): for any infinite sequence $(x_n)_{n=1}^{\infty}$ we have
$$ F_N \geq c \frac{\sqrt{\log{N}}}{N} \qquad \mbox{for infinitely many}~N.$$
Moreover, this is optimal and there are sequences attaining this rate of growth.
\end{proof}

\end{document}